\documentclass[12pt,letterpaper]{amsart}
\usepackage{geometry}
\usepackage{hyperref,amsmath}
\usepackage{mathpazo}
\usepackage{amssymb,url,setspace,xcolor}
\newtheorem{theorem}{Theorem}
\newtheorem{definition}{Definition}
\newtheorem{proposition}{Proposition}

\newtheorem{remark}{Remark}
\onehalfspace

\title[Solvability of Gleason's problem]{Solvability of the Gleason problem on a class of bounded pseudoconvex domains}

\author{Timothy G. Clos}
\address{Timothy G. Clos,\, Kent State University}
\email{ tclos@kent.edu}
\date{\today}

\begin{document}

 \maketitle

\begin{abstract}
We show that if a bounded pseudoconvex domain satisfies the solvability of the bounded $\overline{\partial}$ problem, then the ideal of bounded holomorphic functions vanishing at a point $\alpha$ in the domain is generated by $(z-\alpha)$.  We also prove a smooth analog of the main result for bounded pseudoconvex domains with a sufficiently smooth boundary and also consider the Bergman space case.
\end{abstract}

\section{Preliminaries}
Let $\Omega\subset \mathbb{C}^n$ be a domain and we denote the space of bounded holomorphic functions on $\Omega$ as $H^{\infty}(\Omega)$.  Here, we equip $H^{\infty}(\Omega)$ with the usual compact open topology.  We say a bounded pseudoconvex domain $\Omega\subset \mathbb{C}^n$ is Gleason solvable if the ideal of functions in $H^{\infty}(\Omega)$ vanishing at a point is finitely generated.  More precisely, we have the following definition.
\begin{definition}
We say $\Omega$ Gleason solvable if for any $\alpha\in \Omega$, the ideal in $H^{\infty}(\Omega)$ of functions vanishing at $\alpha$ is generated by $(z-\alpha)$.  That is, if $g\in H^{\infty}(\Omega)$ and $g(\alpha)=0$, then there exists $g_1,...,g_n\in H^{\infty}(\Omega)$ so that $g\equiv \sum_{j=1}^n (z_j-\alpha_j)g_j$ where $\alpha:=(\alpha_1,...,\alpha_n)$.
\end{definition}

As we will show, bounded pseudoconvex domains for which the $\overline{\partial}$-problem is solvable in $L^{\infty}$, which we will call $L^{\infty}$-pseudoconvex, have the nice algebraic property of Gleason solvablility.  As a reference for such domains and more generally Stein manifolds with trivial $L^{\infty}$ $\overline{\partial}$-cohomology, see the classical work \cite{HL}.  We will define $L^{\infty}$-pseudoconvex domains in a more precise manner.

\begin{definition}
Let $\Omega\subset \mathbb{C}^n$ be a bounded pseudoconvex domain.  We say $\Omega$ is $L^{\infty}$-pseudoconvex if for any $\overline{\partial}$-closed bounded form $\beta\in L^{\infty}_{(0,s)}(\Omega)$, there exists $\alpha\in  L^{\infty}_{(0,s-1)}(\Omega)\cap \textbf{dom}(\overline{\partial})$ so that 
\[\overline{\partial}\alpha=\beta.\]

\end{definition}

Gleason solvability is sometimes referred to as Hefer's condition (as developed by Hefer in 1940) for the ring of bounded holomorphic functions on $\Omega$.  For more information on Hefer's condition, see \cite[Chapter 5]{Kr}.  Many of the methods used in the study of Hefer's condition involve the integral representation techniques developed in \cite{Henk} and \cite{rami}.  Gleason solvability was first considered for the ball algebra on the unit ball in $\mathbb{C}^n$.  This was first proved by Leibenson, albeit informally.  See \cite{WR}.  Gleason solvability for holomorphic Bergman spaces, holomorphic mixed norm spaces, and the holomorphic Bloch spaces were studied in \cite{WR}, \cite{YL}, \cite{JO}, \cite{RS}, and \cite{KZ}.  The domains considered were intitially the unit ball, but then were generalized to strongly pseudoconvex domains with $C^2$-smooth boundary.  There is also considerable interest on other function spaces and domains, including the weighted $L^p$ spaces on egg shaped domains, as seen in \cite{GSR}.  See \cite{Hu} for some results on Gleason solvability on the harmonic Bloch spaces of bounded strongly pseudoconvex domains with smooth boundary.

An application of Gleason solvable domains is a several variables partial generalization of the following commuting Toeplitz operator theorem seen in \cite{ACR}.  We state the theorem for the convenience of the reader.
\begin{theorem}\cite{ACR}
Let $\Omega\subset \mathbb{C}$ be a bounded domain.  Suppose $\phi\in H^{\infty}(\Omega)$ and $\psi$ be a bounded measurable function so that $T_{\phi}$ and $T_{\psi}$ commute on the Bergman space $A^2(\Omega)$.  Then $\psi$ is holomorphic on $\Omega$. 
\end{theorem}

The commuting Toeplitz operator problem in several variables remains unsolved on the Bergman spaces of bounded pseudoconvex domains.  However, some partial results were obtained on the the ball, strongly pseudoconvex domains in general, and on domains where the $\overline{\partial}$-problem can be solved in $L^{\infty}$.  See \cite{ST} for results concerning the commuting Toeplitz operator problem on the Bergman space of such domains.  In the paper \cite{CI} it was shown that the Gleason solvability condition implies $\Omega$ can be identified with an open subset of its maximal ideal space.  Furthermore, this condition implies the density of an algebra generated by conjugate holomorphic and holomorphic functions in various function spaces.  Such density results are required in the proof of the commuting Toeplitz operator problem in several variables as seen in \cite{ST}.  As an example, any bounded strongly pseudoconvex domain in $\mathbb{C}^2$ with a boundary of class $C^4$ is Gleason solvable in the continuous case.    

\begin{theorem}\cite{KN}
Let $\Omega\subset \mathbb{C}^2$ be bounded and strongly pseudoconvex with a $C^4$ smooth boundary.  Then $\Omega$ is Gleason solvable in the continuous case.  That is, assume $\phi\in H^{\infty}(\Omega)\cap C(\overline{\Omega})$ so that $\phi(\alpha)=0$ for some $\alpha:=(\alpha_1, \alpha_2)\in \Omega$.  Then $\phi\equiv (z_1-\alpha_1)\phi_1+(z_2-\alpha_2)\phi_2$ where $\phi_1, \phi_2$ are in the same space as $\phi$. 
\end{theorem}

We will now define some terms related to the $\overline{\partial}$-Koszul complex.  We let $\Omega\subset \mathbb{C}^n $ be a domain and
we define $\mathcal{C}_{(0,s)}:=L^{\infty}(\Omega)\cap C^{\infty}_{(0,s)}(\Omega)$ where $C^{\infty}_{(0,s)}(\Omega)$ are $(0,s)$-forms that are smooth on $\Omega$.
For $V$ an $m$-dimensional vector space with basis $\{e_1,...,e_m\}$, we define
$\wedge^r V:=\textbf{span}\{e_{j_1}\wedge e_{j_2}\wedge...\wedge e_{j_r}: j_1<...<j_r\}$.  Then we define the tensor product \[\Gamma^{\infty}_{(r,s)}:=\wedge^r V\otimes \mathcal{C}_{(0,s)}\] where $r$ and $s$ are non-negative integers.  Then we define the densely defined unbounded operator \[\overline{\partial}:\{f\in \Gamma^{\infty}_{(r,s)}: \overline{\partial}f\in \Gamma^{\infty}_{(r,s+1)}\}\rightarrow \Gamma^{\infty}_{(r,s+1)}\] as \[\overline{\partial}(e_J\otimes W)=e_J\otimes \overline{\partial}W .\]  We note that a similar definition holds for the $L^2$ and $C^{\infty}(\overline{\Omega})$ setting (where $L^{\infty}$ is replaced with $L^2$ in these definitions).  We will denote the $L^2$ analog of $\Gamma^{\infty}_{(r,s)}$ as $\Gamma^2_{(r,s)}$ and the $C^{\infty}(\overline{\Omega})$ analog as $\Gamma^{C^{\infty}(\overline{\Omega})}_{(r,s)}$.  Next, for any bounded holomorphic mapping $F:=(f_1,...,f_n):\Omega\rightarrow \mathbb{C}^n$ we will define the operator \[\tau_F:\Gamma^{\infty}_{(r+1,s)}\rightarrow \Gamma^{\infty}_{(r,s)}\] as

\begin{enumerate}
    \item $\tau_F(e_j\otimes W)=f_jW$
    \item $\tau_F\overline{\partial}=\overline{\partial}\tau_F$ on $\{f\in \Gamma^{\infty}_{(r,s)}: \overline{\partial}f\in \Gamma^{\infty}_{(r,s+1)}\}$
    \item $\tau_F\tau_F=0$ and $\overline{\partial}^2=0$.
    \item $\tau_F(A\wedge B)=\tau_F(A)\wedge B+(-1)^{|A|}A\wedge \tau_F(B)$ where $|A|$ is the order of $A$ in $\bigcup_{r=0}^m \bigwedge^r V$.
\end{enumerate} 

Introducing more notation, we define $\mathcal{D}_s:=\mathcal{C}_{(0,s)}\cap \textbf{dom}(\overline{\partial})$.

\section{Main Results and Proofs}

One can modify the proofs of \cite[Lemma 1]{ST} and \cite[Lemma 3]{ST} to get the following proposition, which we prove for the convenience of the reader.

\begin{proposition}\label{prop3} Let $\Omega\subset \mathbb{C}^n$ be a bounded $L^{\infty}$-pseudoconvex domain and $F:=(f_1,...,f_n):\Omega\rightarrow \Omega$ be injective, holomorphic, and $C^{\infty}$-smooth up to $\overline{\Omega}$.  Let $W\in \Gamma^{\infty}_{(r,s)}\cap \textbf{dom}(\overline{\partial})$ so that
\begin{enumerate}
\item $\overline{\partial}W=0$.
\item $W$ is supported away from $F^{-1}(0,0,...,0)$.
\item $\tau_F(W)=0$.

\end{enumerate}
Then there exists $Y\in \Gamma^{\infty}_{(r+1,s-1)}\cap \textbf{dom}(\overline{\partial})$ so that $\tau_F\overline{\partial}Y=W$.

\end{proposition}

\begin{proof} 

Let $\chi\in C^{\infty}(\overline{\Omega})$ so that $\chi\equiv 1$ on the support of $W$ and $\text{supp}(\chi)\cap F^{-1}(0,0,...,0)=\emptyset$.  Now define
\[g_j:=\frac{\chi \overline{f_j}}{\sum_{j=1}^n |f_j|^2}\] and define
\[X:=\sum_{j=1}^n e_j\otimes g_j\in \Gamma^{\infty}_{(1,0)}.\]  Now, $\tau_F(X)=1$ on the support of $W$ and so we define
$Y:=X\wedge W\in \Gamma^{\infty}_{(r+1,s)}$.  Then 
\[\tau_F(Y)=\tau_F(X)\wedge W-X\wedge \tau_F(W)=1\wedge W=W.\] and the support of $Y$ is away from $F^{-1}(0,0,...,0)$.  Furthermore, $\overline{\partial}Y$ is bounded since $\overline{\partial}W=0$ and $g_j\in C^{\infty}(\overline{\Omega})$ for $j\in \{1,...,n\}$.  Note that in \cite[Proof of Lemma 3]{ST} we needed $X$ to have compact support.  This is not needed in this proposition since we assumed $F$ is smooth up to the closure of the domain (that is, $\{f_1,...,f_n\}\subset C^{\infty}(\overline{\Omega})$) and injective (so $F^{-1}(0,0,...,0)$ is a singleton set).  The rest of the proof follows the descending induction procedure seen in \cite[Proof of Lemma 3]{ST} with some minor modifications.  Let $s=n$, $0\leq r\leq m-1$ and also $W\in \Gamma^{\infty}_{(r,n)}$ so that
$\text{supp}(W)\cap F^{-1}(0)=\emptyset$ and $\tau_F(W)=0$.  It is clear that $\overline{\partial}W=0$ since any $(0,n)$ form is $\overline{\partial}$-closed.  Then there is $Y_1\in \Gamma^{\infty}_{(r+1,n)}$ so that $\tau_F(Y_1)=W$.  Also, it is clear that $\overline{\partial}Y_1=0$.  Since $\Omega$ is $L^{\infty}$-pseudoconvex, there exists $Y\in \Gamma^{\infty}_{(r+1, n-1)}$ so that $\overline{\partial}Y=Y_1$.  In other words,
$\tau_F\overline{\partial}(Y)=W$.  This is our base case.  Now assume our proposition is true for $s=k+1, k+2,...,n$ and $r=0,1,...,m-1$. Suppose $0\leq r\leq m-1$ and suppose $W\in \Gamma_{(r,k)}^{\infty}$ has the following properties:
\begin{enumerate}
\item $\text{supp}(W)\cap F^{-1}(0,0,...,0)=\emptyset$,
\item $\overline{\partial}W=0$ and $\tau_F(W)=0$.
\end{enumerate}

Then there exists $Y_1\in \Gamma^{\infty}_{(r+1,k)}$ so that
\begin{enumerate}
\item $\overline{\partial}Y_1\in \Gamma^{\infty}_{(r+1,k+1)}$,
\item $\text{supp}(Y_1)\cap F^{-1}(0,0,...,0)=\emptyset$
\end{enumerate}

Therefore, $\overline{\partial}Y_1\in \Gamma^{\infty}_{(r+1,k+1)}$ satisfies the conditions of the proposition for $s=k+1$.  In other words,

\begin{enumerate}
\item $\text{supp}(\overline{\partial}Y_1)\cap F^{-1}(0,0,...,0)=\emptyset$,
\item $\overline{\partial}\overline{\partial}Y_1=0$ and $\tau_F\overline{\partial}(Y_1)=\overline{\partial}W=0$.
\end{enumerate}

By the induction hypothesis, we have $\overline{\partial}Y_2\in \Gamma^{\infty}_{(r+2,k+1)}$ and $\tau_F\overline{\partial}Y_2=\overline{\partial}Y_1$ for some $Y_2\in \Gamma^{\infty}_{(r+2,k)}$.  Thus we have \[\overline{\partial}(\tau_F Y_2)=\overline{\partial}Y_1.\]  Let us define $Y_3:=Y_1-\tau_FY_2$.  Then we have
\[\tau_FY_3=W\] and \[\overline{\partial}Y_3=0.\]  Since $\Omega$ is $L^{\infty}$-pseudoconvex,
there exists $Y\in \Gamma^{\infty}_{(r+1, k-1)}\cap\textbf{dom}(\overline{\partial})$ so that $\overline{\partial}Y=Y_3$.  Thus we have 
\[\tau_F\overline{\partial}Y=W\] as desired.

\end{proof}

\begin{remark}\label{rem1}

The same procedure is applicable to the space $\Gamma^2_{(r,s)}$ where the proof of Proposition \ref{prop3} is modified slightly to accomodate the different space.  In this case, we do not need a bounded $L^{\infty}$-pseudoconvex domain, as the $\overline{\partial}$-problem is solvable on $L^2(\Omega)$ for any bounded pseudoconvex domain.  See \cite{EJS} for more information on the $L^2$-theory of the $\overline{\partial}$-Neumann problem.  
\end{remark}

The proof for the following Proposition is the same as for Proposition \ref{prop3}, except one must use the following theorem in place of the solvability of $\overline{\partial}$ in $L^{\infty}$.
\begin{theorem}\cite{ChS}
Let $\Omega\subset \mathbb{C}^n$ be a bounded pseudoconvex domain with a $C^{\infty}$-smooth boundary.  Suppose $\beta\in C^{\infty}_{(0,s)}(\overline{\Omega}) $ so that $\overline{\partial}\beta=0$.  Then there exists $\alpha\in C^{\infty}_{(0,s-1)}(\overline{\Omega}) $ so that $\overline{\partial}\alpha=\beta$.  

\end{theorem}

\begin{proposition}\label{prop4} Let $\Omega\subset \mathbb{C}^n$ be a bounded pseudoconvex domain with a $C^{\infty}$-smooth boundary and $F:=(f_1,...,f_n):\Omega\rightarrow \Omega$ be injective, holomorphic, and $C^{\infty}$-smooth up to $\overline{\Omega}$.  Let $W\in \Gamma^{C^{\infty}(\overline{\Omega})}_{(r,s)}$ so that
\begin{enumerate}
\item $\overline{\partial}W=0$.
\item $W$ is supported away from $F^{-1}(0,0,...,0)$.
\item $\tau_F(W)=0$.

\end{enumerate}
Then there exists $Y\in \Gamma^{C^{\infty}(\overline{\Omega})}_{(r+1,s-1)}$ so that $\tau_F\overline{\partial}Y=W$.

\end{proposition}

The following is the main theorem.  As a consequence of this main theorem, we have that the commuting Toeplitz operator problem as seen in \cite{ST} is true on a possibly larger class of bounded pseudoconvex domains.

\begin{theorem}\label{main}
Let $\Omega$ be a bounded $L^{\infty}$-pseudoconvex domain in $\mathbb{C}^n$.  Then $\Omega$ is Gleason solvable.
\end{theorem}

\begin{proof}
This proof is an application of a few results concerning the $\overline{\partial}$-Koszul complex.  See \cite{ST} for more information on the $\overline{\partial}$-Koszul complex.  We will first consider $n=2$.
Let $\Omega\subset \mathbb{C}^2$ be a bounded $L^{\infty}$-pseudoconvex domain.
Let $g\in H^{\infty}(\Omega)$ and without loss of generality, suppose $g(0,0)=0$.  Since $g$ has a power series expansion about $(0,0)$ converging on some open set $U\subset \subset \Omega$, one can write \[g=z_1\lambda_1+z_2\lambda_2\] on $U$.  Here, $\lambda_1, \lambda_2\in H^{\infty}(U)\cap C^{\infty}(\overline{U})$.  Now let $\chi_1\in C^{\infty}(U)$, supported in $U$, $0\leq \chi_1\leq 1$ on $U$ and $\chi\equiv 1$ on $\widetilde{U}\subset \subset U$.  Also, $(0,0)\in \widetilde{U}$.  Then on $\Omega\setminus \overline{\widetilde{U}}$ we have \[g=z_1\left(g\frac{\overline{z_1}}{|z_1|^2+|z_2|^2}\right)+z_2\left(g\frac{\overline{z_2}}{|z_1|^2+|z_2|^2}\right)\]
Therefore, \[g\equiv z_1\left((1-\chi_1)g\frac{\overline{z_1}}{|z_1|^2+|z_2|^2}+\chi_1\lambda_1\right)+z_2\left((1-\chi_1)g\frac{\overline{z_2}}{|z_1|^2+|z_2|^2}+\chi_1\lambda_2\right)\] on $\Omega$.  Now for the sake of notation, define \[L_1:=\left((1-\chi_1)g\frac{\overline{z_1}}{|z_1|^2+|z_2|^2}+\chi_1\lambda_1\right)\] and 
\[L_2:=\left((1-\chi_1)g\frac{\overline{z_2}}{|z_1|^2+|z_2|^2}+\chi_1\lambda_2\right)\]
Notice that both $L_1$ and $L_2$ are bounded on $\Omega$ and partial derivatives of $L_1$ and $L_2$ of all orders exist and are continuous on compact subsets of $\Omega$.  In addition, $\overline{\partial}L_1$ and $\overline{\partial}L_2$ are bounded on $\Omega$.  Now define 
\[W_1:=e_1\otimes \overline{\partial}L_1+e_2\otimes \overline{\partial}L_2.\]  Notice that $W_1$ satisfies the hypothesis of Proposition \ref{prop3} since the support of $W_1$ is away from $F^{-1}(0,0)$ where $F:=(z_1,z_2)$, $W_1\in \Gamma^{\infty}_{(1,1)}\cap \textbf{dom}(\overline{\partial})$,  $\overline{\partial}W_1=0$, and $\tau_F W_1=0$.  Thus by Proposition \ref{prop3}, there exists there exists $H_1\in \mathcal{C}_{(0,0)}^{\infty}\cap\textbf{dom}(\overline{\partial})$ so that \[Y_1:=\left(e_1\wedge e_2\right)\otimes H_1\in \Gamma^{\infty}_{(2,0)}\] and \[\tau_F\overline{\partial}Y_1=e_2\otimes z_1\overline{\partial}H_1-e_1\otimes z_2\overline{\partial}H_1 =W_1.\]  Thus we have that
\[\overline{\partial}L_1=-z_2\overline{\partial}H_1\] and
\[\overline{\partial}L_2=z_1\overline{\partial}H_1.\]  Therefore, \[L_1+z_2H_1\in H^{\infty}(\Omega)\] and \[L_2-z_1H_1\in H^{\infty}(\Omega).\]  Furthermore, \[z_1(L_1+z_2H_1)+z_2(L_2-z_1H_1)\equiv g\] on $\Omega$.  This completes the proof for $n=2$.  We will demonstrate the proof works for $n=3$ and then show how it can be generalized.  For $\Omega\subset   \mathbb{C}^n$ for $n=3$, we first use a power series argument similar to the $n=2$ case.  That is, if $g(0)=0$, one can write $g\equiv z_1L_1+z_2L_2+z_3L_3$ where $L_j$ are holomorphic on a neighborhood of $0$, $L_j$ are smooth and bounded on $\Omega$, and $\overline{\partial}L_j\in L^{\infty}_{(0,1)}(\Omega)$ for $j=1,2,3$.  We define \[W:=e_1\otimes \overline{\partial}L_1+e_2\otimes \overline{\partial}L_2+e_3\otimes \overline{\partial}L_3\in \Gamma^{\infty}_{(1,1)}.\]  Clearly, $\overline{\partial}W=0$ and $\tau_F(W)=0$ where $F=(z_1,z_2,z_3)$.  Furthermore, the support of $W$ is away from $F^{-1}(0,0,0)$.  Therefore, by Proposition \ref{prop3}, there exists $Y\in \Gamma^{\infty}_{(2,0)}$ so that 
\[\tau_F\overline{\partial}Y=W.\]  Now, $Y\in \Gamma^{\infty}_{(2,0)}$, therefore, $Y$ has the form
\[Y=e_1\wedge e_2\otimes y_3+e_1\wedge e_3\otimes y_2+e_2\wedge e_3\otimes y_1\] for $\{y_1, y_2, y_3\}\subset \mathcal{C}_{(0,0)}^{\infty}\cap\textbf{dom}(\overline{\partial})$.  Now we apply $\tau_F\overline{\partial}$ to $Y$ and collect terms together.  We get
\[W=\tau_F\overline{\partial}Y=e_1\otimes \left(-z_2\overline{\partial}y_3-z_3\overline{\partial}y_2\right)+           e_2\otimes\left(z_1\overline{\partial}y_3-z_3\overline{\partial}y_1\right)+e_3\otimes \left(z_1\overline{\partial}y_2+z_2\overline{\partial}y_1\right).\]  Thus we have
\[\overline{\partial}L_1=\left(-z_2\overline{\partial}y_3-z_3\overline{\partial}y_2\right),\]
\[\overline{\partial}L_2=\left(z_1\overline{\partial}y_3-z_3\overline{\partial}y_1\right),\]
and
\[\overline{\partial}L_3=\left(z_1\overline{\partial}y_2+z_2\overline{\partial}y_1\right).\]
This implies \[L_1+z_2y_3+z_3y_2\in H^{\infty}(\Omega),\]
\[L_2-z_1y_3+z_3y_1\in H^{\infty}(\Omega),\]
and
\[L_3-z_1y_2-z_2y_1\in H^{\infty}(\Omega).\]  Furthermore, we can write
\[g\equiv z_1\left(L_1+z_2y_3+z_3y_2\right)+z_2\left(L_2-z_1y_3+z_3y_1\right)+z_3\left(L_3-z_1y_2-z_2y_1\right).\]  The proof for $n>3$ is an analog of the previous proofs with more terms involved.  We write $g$ as a linear combination of bounded functions in the domain of $\overline{\partial}$ that are holomorphic on a neighborhood of $0$.  Then we use Proposition \ref{prop3} to 'correct' these bounded functions to be holomorphic and bounded on $\Omega$.  Furthermore, the linear combination of these functions still gives $g$.
\end{proof}
The advantage of this approach is that one can modify the proof to consider the Bergman space $A^2(\Omega)$. Recall the Bergman space of $\Omega$ is the space of square integrable holomorphic functions on $\Omega$.  The proof of this following theorem follows the proof of Theorem \ref{main} where Proposition \ref{prop3} is used in addition to Remark \ref{rem1}.
\begin{theorem}
Let $\Omega\subset\mathbb{C}^n$ be a bounded pseudoconvex domain.  Suppose $\phi\in A^2(\Omega)$ and $\phi(\alpha)=0$ for some $\alpha\in \Omega$.  Then, there exists $\phi_1,\phi_2,...,\phi_n\in A^2(\Omega)$ so that $\phi\equiv\sum_{j=1}^n(z_j-\alpha_j)\phi_j$.
\end{theorem}

It is well known that the boundary smoothness of the domain plays a role in determining the regularity of the solutions of the $\overline{\partial}$-problem with smooth data. Therefore, one can prove the following result for the $\overline{\partial}$ problem with $C^{\infty}$-smooth data using the techniques developed for the $L^{\infty}$ case (see the proof of Theorem \ref{main}) together with Propostion \ref{prop4}.  This result is a variation of one of the main results in \cite{KN}.  The proof follows the proof of Theorem \ref{main} with Proposition \ref{prop4} replacing Proposition \ref{prop3}.

\begin{theorem}\label{thm6}
Let $\Omega\subset \mathbb{C}^n$ be a $C^{\infty}$-smooth bounded pseudoconvex domain.  Then $\Omega$ has solvable $C^{\infty}$-smooth Gleason problem.  That is, for any $\phi\in C^{\infty}(\overline{\Omega})\cap A^2(\Omega)$ vanishing at $\alpha:=(\alpha_1,...,\alpha_n)\in \Omega$, there exists
\[\{\phi_1,...,\phi_n\}\subset C^{\infty}(\overline{\Omega})\cap A^2(\Omega)\] so that
\[\phi\equiv \sum_{j=1}^n (z_j-\alpha_j)\phi_j.\] 
\end{theorem}

\section{Aknowlegments}
I wish to thank Akaki Tikaradze and Alexander Izzo for useful conversations and comments on a preliminary version of this manuscript.  I also thank the anonymous referees for their suggestions. 

\section{Data Availability}
Data sharing not applicable to this article as no datasets were generated or analysed during the current study.

\bibliographystyle{amsalpha}
\bibliography{refscompthree}

\end{document}